\RequirePackage{etoolbox}
\csdef{input@path}{{style/}{graphics/}}
\makeatletter
\input{arxiv-vmsta.cfg}
\makeatother
\documentclass[numbers,compress]{vmsta}
\usepackage{mathbh}
\usepackage{graphicx}

\volume{1}
\pubyear{2014}
\firstpage{195}
\lastpage{209}
\doi{10.15559/15-VMSTA19}

\startlocaldefs

\newcommand{\rrvert}{\vert}
\newcommand{\llvert}{\vert}
\allowdisplaybreaks

\newcommand{\PP}{{\mathcal P}}
\newcommand{\R}{{\mathbb R}}
\newcommand{\XXX}{\mathfrak{X}}
\newcommand{\FFF}{\mathfrak{F}}
\newcommand{\Ii}{\mathbh{1}}

\DeclareMathOperator{\bcdot}{\boldsymbol{\cdot}}
\DeclareMathOperator{\pr}{\mathsf P}
\DeclareMathOperator{\M}{\mathsf E} \DeclareMathOperator{\D}{Var}
\DeclareMathOperator{\cov}{Cov}

\newtheorem{thm}{Theorem}
\theoremstyle{definition}
\newtheorem{example}{Example}
\newtheorem{experiment}{Experiment}
\endlocaldefs

\begin{document}
\begin{frontmatter}

\title{Testing hypotheses on moments by observations
from~a~mixture with varying concentrations}
\author{\inits{A.}\fnm{Alexey}\snm{Doronin}}\email{al\_doronin@ukr.net}
\author{\inits{R.}\fnm{Rostyslav}\snm{Maiboroda}}\email{mre@univ.kiev.ua}

\address{Kyiv National Taras Shevchenko University, Kyiv, Ukraine}

\markboth{A. Doronin, R. Maiboroda}{Testing hypotheses on moments by
observations
from a mixture with varying concentrations}

\begin{abstract}
A mixture with varying concentrations is a modification of a finite
mixture model in which the mixing probabilities (concentrations
of mixture components) may be different for different observations.
In the paper, we assume that the concentrations are
known and the distributions of components are completely unknown.
Nonparametric technique is proposed for testing hypotheses on
functional moments of components.
\end{abstract}


\received{22 December 2014}
%
\revised{21 January 2015}
%
\accepted{22 January 2015}
\publishedonline{4 February 2015}
\end{frontmatter}

\section{Introduction}
Finite mixture models (FMMs) arise naturally in statistical
analysis of biological and sociological data
\cite{McLachlan,Titterington:FiniteMixture}.
The model of mixture with varying
concentrations (MVC) is a modification of the FMM where the
mixing probabilities may be different for different observations.
Namely, we consider a sample of subjects $O_1,\dots,O_N$ where
each subject belongs to one of subpopulations (mixture components)
$\PP_1, \dots,\allowbreak\PP_M$. The true subpopulation to which the
subject $O_j$ belongs is unknown, but we know the probabilities
$p_{j;N}^m=\pr[O_j\in\PP_m]$ (mixing probabilities,
concentrations
of $\PP_m$ in the mixture at the $j$th observation, $j=1,\dots,N$,
$m=1,\dots,M$). For each
subject $O$, a variable $\xi(O)$ is observed, which is considered
as a random element in a measurable space $\XXX$
equipped by a $\sigma$-algebra $\FFF$. Let
\[
F_m(A)=\pr\bigl[\xi(O)\in A\ |\ O\in\PP_m\bigr],\quad A\in
\FFF,
\]
be the distribution of $\xi(O)$ for subjects $O$ that belong to
the $m$th component. Then the unconditional distribution of
$\xi_{j;N}=\xi(O_j)$ is
\begin{equation}
\label{Eq1} \pr[\xi_{j;N}\in A]=\sum_{m=1}^M
p_{j;N}^m F_m(A),\quad A\in\FFF.
\end{equation}
The observations $\xi_{j;N}$ are assumed to be independent for
$j=1,\dots,N$.

We consider the nonparametric MVC model where the concentrations
$p_{j;N}^m$ are known but the component distributions $F_m$
are completely unknown. Such models were applied to analyze
gene expression level data
\cite{Maiboroda:StatisticsDNA} and data on
sensitive questions in sociology \cite{Shcherbina}. An example of
sociological data
analysis based on MVC is presented in \cite{Maiboroda:AdaptMVC}.
In this paper, we consider adherents of different political parties in
Ukraine as subpopulations $\PP_i$. Their concentrations
are deduced from 2006 parliament election results in different
regions of Ukraine. Individual voters are considered as
subjects; their observed characteristics are taken from the
Four-Wave Values Survey held in Ukraine in 2006. (Note that the
political choices of the surveyed individuals were unknown.
So, each subject must be considered as selected from mixture
of different $\PP_i$.) For example, one of the observed characteristics
is the satisfaction of personal income (in points from 1 to 10).

A natural question in the analysis of such data is
homogeneity testing for different components. For example, if $\XXX=\R
$, then we
may ask if the means or variances (or both) of the distributions
$F_i$ and $F_k$ are the same for some fixed~$i$ and $k$ or if
the variances of all the components are the same.

In \cite{Maiboroda:StatisticsDNA}, a test is proposed for
the hypothesis of two-means homogeneity. In this paper, we
generalize the approach from \cite{Maiboroda:StatisticsDNA} to a
much richer class of hypotheses, including different
statements on means, variances, and other generalized functional
moments of component distributions.

Hypotheses of equality of MVC component distributions, that is,
$F_i\equiv F_k$, were considered in \cite{Maiboroda2000} (a
Kolmogorov--Smirnov-type test is proposed) and~
\cite{Autin:TestDensities} (tests based on wavelet density
estimation). The
technique of our paper also allows testing such
hypotheses using the ``grouped $\chi^2$''-approach.

Parametric tests for different hypotheses on mixture components
were also considered in
\cite{Johnson,Liu,Titterington:FiniteMixture}.

The rest of the paper is organized as follows. We describe the considered
hypotheses formally and discuss the test construction in
Section \ref{SectHypothesis}. Section \ref{SectEsim}
contains auxiliary information on the functional moments
estimation in MVC models. In Section~\ref{SectTest}, the test is
described formally.
Section \ref{SectNum} contains results of the test
performance analysis by a simulation study and an example of real-life
data analysis.
Technical proofs are
given in Appendix~\ref{appendix}.

\section{Problem setting}\label{SectHypothesis}

In the rest of the paper, we use the following notation.

The zero vector from $\mathbb{R}^k$ is denoted by $\mathbb{O}_k$.
The unit $k\times k$-matrix is denoted by $\mathbb{I}_{k\times k}$,
and the $k\times m$-zero matrix by $\mathbb{O}_{k\times m}$.
Convergences in probability and in distribution are denoted
$\stackrel{P}{\longrightarrow}$ and
$\stackrel{d}{\longrightarrow}$, respectively.

We consider the set of concentrations $p=(p_{j;N}^m, j=1,\dots,N;\
m=1,\dots,M;\allowbreak N=1,\dots)$
as an infinite array, $p_{\bcdot;N}^{\bcdot}=(p_{j;N}^m, j=1,\dots,N;\
m=1,\dots,M)$
as an $(N\times m)$-matrix, and $p_{\bcdot;N}^m=(p_{j;N}^m,
j=1,\dots,N)\in\R^d$ and $p_{j,N}^{\bcdot}=(p_{j;N}^m,m=1,\dots,M)$
as column
vectors. The same notation is used for arrays of similar
structure, such as the array $a$ introduced further.

Angle brackets with subscript $N$ denote averaging
of an array over all the observations, for example,
\[
\bigl\langle a_{\bcdot;N}^m\bigr\rangle_N=
{1\over N}\sum_{j=1}^N
a_{j;N}^m.
\]
Multiplication, summation, and other similar operations inside
the angle brackets are applied to the arrays componentwise,
so
that
\[
\bigl\langle a_{\bcdot;N}^m p_{\bcdot;N}^k\bigr
\rangle_N= {1\over N}\sum_{j=1}^N
a_{j;N}^m p_{j;N}^k,\qquad \bigl\langle
\bigl(a_{\bcdot;N}^m\bigr)^2 \bigr
\rangle_N= {1\over N}\sum_{j=1}^N
\bigl(a_{j;N}^m \bigr)^2,
\]
and so on.

Angle brackets without subscript mean the limit of the
corresponding averages as $N\to\infty$ (assuming that this limit
exists):
\[
\bigl\langle p^m a^k\bigr\rangle=\lim
_{N\to\infty} \bigl\langle p_{\bcdot;N}^m
a_{\bcdot;N}^k\bigr\rangle_N.
\]
We introduce formally random elements $\eta_m\in\XXX$ with
distributions
$F_m$, $m=\allowbreak 1,\dots, M$.

Consider a set of $K \leq M$ measurable functions
$g_k:\XXX\to\R^{d_k}$,
$k=1\xch{,\dots,}{,\dots} K$. Let $\bar{g}_k^m$ be the (vector-valued) functional
moment of the $m$th component with moment function $g_k$, that is,
\begin{equation}
\label{Eq2} \bar{g}_k^m := \M\bigl[ g_k(
\eta_m) \bigr] \in\mathbb{R}^{d_k}.
\end{equation}

Fix a measurable function
$T:\R^{d_1}\times\R^{d_2}\times\cdots\times\R^{d_K}\to\R^L$.
For data described by the MVC model (\ref{Eq1}) we consider testing a
null-hypothesis
of the form
\begin{equation}
\label{Eq3} H_0: T\bigl(\bar{g}_1^1,\dots,
\bar{g}_K^K\bigr) = \mathbb{O}_L
\end{equation}
against the general alternative
$T(\bar{g}_1^1,\dots,\bar{g}_K^K)\not= \mathbb{O}_L$.

\begin{example} Consider a three-component mixture ($M=3$)
with $\XXX=\R$. We would like to test the hypothesis
$H_0^{\sigma}:\D\eta_1=\D\eta_2$ (i.e., the variances of the first and
second components are the same). This hypothesis can be
reformulated in the form (\ref{Eq3}) by letting
$g_1(x)=g_2(x)=(x,x^2)^T$ and
$T((y_{11},y_{12})^T,(y_{21},y_{22})^T)=
(y_{12}-(y_{11})^2,y_{22}-(y_{21})^2)^T$.
\end{example}

\begin{example} Let $\XXX=\R$. Consider the hypothesis of mean homogeneity
$H_0^\mu: \M\eta_1=\cdots=\M\eta_M$. Then the choice of
$g_m(x)=x$,
$T(y_1,\dots,y_M)=(y_1-y_2,y_2-y_3,\dots,y_{M-1}-y_M)^T$ reduces
$H_0^\mu$ to the form (\ref{Eq3}).
\end{example}

\begin{example} Let $\XXX$ be a finite discrete space:
$\XXX=\{x_1,\dots,x_r\}$. Consider the distribution homogeneity
hypothesis $H_0^{\equiv}: F_1\equiv
F_2$. To present it in the form (\ref{Eq3}), we can use
$g_i(x)=(\Ii\{x=x_i\},k=1,\dots,r-1)^T$ and
$T(y_1,y_2)=y_1-y_2$ ($y_i\in\R^{r-1}$\vadjust{\eject} for $i=1,2$). In the case
of continuous distributions, $H_0^{\equiv}$ can be discretized by
data grouping.
\end{example}

To test $H_0$ defined by (\ref{Eq3}), we adopt the following
approach. Let there be some consistent estimators $\hat g_{k;N}^m$ for
$\bar g_k^m$. Assume that $T$ is continuous. Consider the statistic
$\hat T_N=T(\hat g_{1;N}^1,\dots, \hat g_{K;N}^K)$.
Then, under $H_0$, $\hat T_N\approx\mathbb{O}_L $, and a
far departure of $\hat T_N$ from zero will evidence in favor of
the alternative.

To measure this departure, we use a Mahalanobis-type
distance. If $\sqrt{N}\hat T_N$ is asymptotically normal with a
nonsingular asymptotic covariance matrix~$D$, then, under $H_0$,
$N\hat T_N^T D^{-1}\hat T_N$ is asymptotically $\chi^2$-distributed. In
fact, $D$ depends on unknown component distributions $F_i$, so
we replace it by its consistent estimator $\hat D_N$. The
resulting statistic $\hat s_N=N\hat T_N^T \hat D_N^{-1}\hat T_N$ is
a test statistic. The test rejects $H_0$ if
$\hat s_N>Q^{\chi^2_L}(1-\alpha)$, where $\alpha$ is the significance level,
and $Q^G(\alpha)$ denotes the quantile of level $\alpha$ for
distribution $G$.

Possible candidates for the role of estimators $\hat g_{k;N}^m$ and
$\hat D_N$ are considered in the next section.

\section{Estimation of functional moments}\label{SectEsim}

Let us start with the nonparametric estimation of $F_m$ by the
weighted empirical distribution of the form
\[
\hat F_{m;N}(A)={1\over N}\sum
_{j=1}^N a_{j;N}^m\Ii\{\xi_{j;N}\in A\},
\]
where $a_{j;N}^m$ are some nonrandom weights to be selected
``in the best way.'' Denote $e_m=(\Ii\{k=m\}, k=1,\dots,M)^T$ and
\[
\varGamma_N={1\over N} \bigl(p_{\bcdot;N}^{\bcdot}
\bigr)^Tp_{\bcdot;N}^{\bcdot} =\bigl(\bigl\langle
p_{\bcdot;N}^m p_{\bcdot;N}^i\bigr
\rangle_N\bigr)_{m,i=1}^M.
\]
Assume that $\varGamma_N$ is nonsingular.
It is shown in \cite{Maiboroda:StatisticsDNA}
that, in this case, the weight array
\[
a_{\bcdot;N}^m=p_{\bcdot;N}^{\bcdot}
\varGamma_N^{-1}e_m
\]
yields the unbiased estimator with
minimal assured quadratic risk.

The \textit{simple} estimator $\hat g_{i;N}^m$ for $\bar g_i^m$ is
defined as
\[
\hat g_{i;N}^m=\int_{\XXX}g_i(x)
\hat F_{m;N}(dx) ={1\over N}\sum
_{j=1}^N a_{j;N}^m
g_i(\xi_{j;N}).
\]
We denote
$\varGamma=\lim_{N\to\infty}\varGamma_N=(\langle
p^ip^m\rangle)_{i,m=1}^M$. Let $h:\XXX\to\R^d$ be any measurable
function.
\begin{thm}\label{Th1} {{\normalfont(\cite{Maiboroda:AdaptMVC}, Lemma 1)}}
Assume that:
\begin{enumerate}
\item[\emph{(i)}] $\varGamma$ exists, and $\det\varGamma\not=0$;

\item[\emph{(ii)}] $\M[|h(\eta_m)|]< \infty, \ m=1\xch{,\dots,}{,\dots} M$.
\end{enumerate}

Then $\hat h_N^m\stackrel{P}{\longrightarrow}\M[h(\eta_m)]$ as
$N\to\infty$ for all $m=1,\dots,M$.\vadjust{\eject}
\end{thm}

To formulate the asymptotic normality result for the simple moment
estimators, we need some additional notation.

We consider the set of all moments $\bar g_k^k$, $k=1,\dots,K$,
as one long vector belonging to $\mathbb{R}^d$, $d := d_1+\cdots+d_K$:
\begin{subequations}
\begin{equation}
\bar{g} := \bigl( \bigl(\bar{g}_1^1\bigr)^T,
\dots,\bigl(\bar{g}_K^K\bigr)^T
\bigr)^T \in \mathbb{R}^d .
\end{equation}
The corresponding estimators also form a long vector
\begin{equation}
\hat{g}_N := \bigl( \bigl(\hat{g}_{1;N}^1
\bigr)^T,\dots,\bigl(\hat{g}_{K;N}^K
\bigr)^T \bigr)^T \in\mathbb{R}^d .
\end{equation}
\end{subequations}
We denote the matrices of mixed second moments of $g_k(x)$, $k=1,\dots,K$,
and the corresponding estimators
as
\begin{subequations}
\begin{align}
\bar{g}_{k,l}^m &:= \M\bigl[ g_k(
\eta_m) g_l(\eta_m)^T \bigr] \in
\mathbb{R}^{d_k \times d_l}, \quad k,l=1,\dots,K,\ m=1,\dots,M ;\\
\label{EqHatgkl} \hat{g}_{k,l;N}^m &:= \frac{1}{N} \sum
_{j=1}^N a_{j;N}^m
g_k(\xi_{j;N}) g_l(\xi_{j;N})^T
\in\mathbb{R}^{d_k \times d_l}, \quad k,l=1,\dots,K.
\end{align}
\end{subequations}

We consider the function $T$ as a function of $d$-dimensional
argument, that~is, $T(y) := T(y^1,\dots,y^K)$. Then
$\hat{T}_N := T(\hat{g}_N) =
T(\hat{g}_{1;N}^1,\dots,\hat{g}_{K;N}^K)$.

Let us define the following matrices (assuming that the limits
exist):
\begin{subequations}
\begin{align}
\alpha_{r,s;N} &:= \bigl(\alpha_{r,s;N}^{k,l}
\bigr)_{k,l=1,\dots,K} := \bigl( \bigl\langle a_{\cdot;N}^ka_{\cdot;N}^lp_{\cdot;N}^rp_{\cdot;N}^s
\bigr\rangle_N \bigr)_{k,l=1,\dots,K} \in\mathbb{R}^{K\times K} ;\\
\alpha_{r,s} &:= \bigl(\alpha_{r,s}^{k,l}
\bigr)_{k,l=1,\dots,K} := \Bigl( \lim_{N\to\infty} \alpha_{r,s;N}^{k,l}
\Bigr)_{k,l=1,\dots,K} \in\mathbb{R}^{K\times K}, \quad r,s=1,\dots,M ;
\end{align}\vspace*{-12pt}
\end{subequations}
\begin{subequations}
\begin{align}
\beta_{m;N}& := \bigl( \beta_{m;N}^{k,l}
\bigr)_{k,l=1,\dots,K} := \bigl( \bigl\langle a_{\cdot;N}^ka_{\cdot;N}^lp_{\cdot;N}^m
\bigr\rangle_N \bigr)_{k,l=\overline{1,K}} \in\mathbb{R}^{K\times K} ;\\
\beta_{m}& := \bigl( \beta_{m}^{k,l}
\bigr)_{k,l=1,\dots,K} := \Bigl( \lim_{N\to\infty} \beta_{m;N}^{k,l}
\Bigr)_{k,l=1,\dots,K} \in\mathbb{R}^{K\times K}, \quad m=1,\dots M .
\end{align}
\end{subequations}

Then the asymptotic covariance matrix of the normalized estimate
$\sqrt{N}(\hat g_N-\bar g)$ is $\varSigma$, where $\varSigma$ consists
of the blocks $\varSigma^{(k,l)}$:
\begin{subequations}
\begin{align}
\varSigma^{(k,l)}& := \sum_{m=1}^M
\beta_m^{k,l} \bar{g}_{k,l}^m - \sum
_{r,s=1}^M \alpha_{r,s}^{k,l}
\bar{g}_k^r \bigl(\bar{g}_l^s
\bigr)^T \in \mathbb{R}^{d_k \times d_l} ;\\
\varSigma&:= \bigl(\varSigma^{(k,l)} \bigr)_{k,l=1,\dots,K} \in
\mathbb{R}^{d \times d}.
\end{align}
\end{subequations}

\begin{thm}\label{Th2}

Assume that:
\begin{enumerate}
\item[\emph{(i)}] The functional moments $\bar{g}_k^m$, $\bar{g}_{k,l}^m$ exist
and are finite for $k,l=1,\dots,K$, $m=1,\dots,M$.

\item[\emph{(ii)}] There exists $\delta> 0$ such that $\M [ |g_k(\eta
_m)|^{2+\delta}  ] < \infty$,
$k=1,\dots,K$, $m=1,\dots,M$.

\item[\emph{(iii)}] There exist finite matrices $\varGamma$, $\varGamma^{-1}$, $\alpha
_{r,s}$, and $\beta_m$
for $r,s,m=1,\dots,M$.
\end{enumerate}

Then
$\sqrt{N} (\hat{g}_N - \bar{g}) \stackrel{d}{\longrightarrow} \zeta
\simeq\mathcal{N}(\mathbb{O}_d, \varSigma)$, $N\to\infty$.\vadjust{\eject}
\end{thm}

Thus, to construct a test for $H_0$, we need a consistent estimator for
$\varSigma$. The matrices $\alpha_{r,s;N}$ and $\beta_{m;N}$ are natural
estimators for $\alpha_{r,s}$ and $\beta_{m}$. It is also natural
to estimate $\bar{g}_{k,l}^m$ by $\hat{g}_{k,l;N}^m$ defined in
(\ref{EqHatgkl}). In view of Theorem~\ref{Th1},
these estimators are consistent under the assumptions of Theorem \ref{Th2}.
But they can possess undesirable properties for moderate sample
size. Indeed, note that $\hat F_{m;N}$ is not a
probability distribution itself since the weights $a_{j;N}^m$
are negative for some $j$. Therefore, for example, the simple
estimator of
the second moment of some component can be negative, estimator
(\ref{EqHatgkl})
for the positive semidefinite matrix $\bar{g}_{k,k}^m$ can
be not positive semidefinite matrix, and so on. Due to the asymptotic
normality result, this is not too troublesome for estimation of
$\bar g$. But it causes serious difficulties when one uses an
estimator of the asymptotic covariance matrix $D$ based on $\hat{g}_{k,l;N}^m$
in order to calculate $\hat s_N$.

In \cite{MK2005}, a technique is developed of $\hat F_{m;N}$ and
$\hat h_{N}^m$ improvement that allows one to derive estimators with
more adequate finite sample properties if
$\XXX=\R$.

So, assume that $\xi(O)\in\R$ and consider the weighted empirical
CDF
\[
\hat F_{m;N}(x)={1\over N}\sum
_{j=1}^N a_{j;N}^m\Ii\{
\xi_{j;N}<x\}.
\]
It is not a nondecreasing function, and it can attain values
outside $[0,1]$ since some $a_{j;N}^m$ are negative. The
transform
\[
\tilde F_{m;N}^{+}(x)=\sup_{y<x}\hat
F_{m;N}(y)
\]
yields a monotone function $\tilde F_{m;N}(x)$, but it still can
be greater than 1 at some~$x$. So, define
\[
\hat F_{m;N}^{+}(x)=\min\bigl\{1,\tilde F^{+}_{m;N}(x)
\bigr\}
\]
as the improved estimator for $F_m(x)$. Note that this is an
``improvement upward,'' since
$\tilde F_{m;N}^{+}(x)\ge\hat F_{m;N}(x)$. Similarly, a
downward improved estimator can be defined as
\begin{align*}
\tilde F_{m;N}^{-}(x)&=\inf_{y\ge x}\hat
F_{m;N}(y),\\
\hat F_{m;N}^{-}(x)&=\max\bigl\{0,\tilde F^{-}_{m;N}(x)
\bigr\}.
\end{align*}
Any CDF that lies between $\hat F_{m;N}^{-}(x)$ and $\hat
F_{m;N}^{+}(x)$ can be considered as an~improved version of
$\hat F_{m;N}(x)$. We will use only one such improvement, which
combines\allowbreak $\hat F_{m;N}^{-}(x)$ and $\hat F_{m;N}^{+}(x)$:
\begin{equation}
\label{Ku3:Eq10} \hat F_{m;N}^{\pm}(x)= %
\begin{cases}
\hat F_{m;N}^{+}(x) &\text{if $ \hat F_{m;N}^{+}(x,a)\le1/2 $,}\\
\hat F_{m;N}^{-}(x) &\text{if $ \hat F_{m;N}^{-}(x,a)\ge1/2 $,}\\
1/2 &\text{otherwise.}
\end{cases}
\end{equation}

Note that all the three considered estimators $\hat F^*_{m;N}$
($*$ means any symbol from $+$, $-$, or $\pm$) are piecewise
constants on intervals between successive order statistics of the
data. Thus, they can be represented as
\[
\hat F_{m;N}^*(x)={1\over N}\sum
_{j=1}^N b_{j;N}^{m*}\Ii\{
\xi_{j;N}<x\},
\]
where $b_{j;N}^{m*}$ are some random weights that depend on the data.

The corresponding \textit{improved} estimator for $\bar g_i^m$ is
\[
\hat g_{i;N}^{m*}=\int_{-\infty}^{+\infty}
g_i(x)\hat F_{m;N}^*(dx) ={1\over N}\sum
_{j=1}^N b_{j;N}^*
g_i(\xi_{j;N}).
\]
Let $h:\R\to\R$ be a measurable function.

\begin{thm}\label{Th3}
Assume that $\varGamma$ exists and $\det\varGamma\not=0$.
\begin{enumerate}
\item[\emph{(I)}] If for some $c_{-}<c_{+}$, $c_{-}\le\eta_m\le
c_{+}$ for all $m=1,\dots, M$ and $h$ has bounded variation on
$(c_{-},c_{+})$, then
$\hat h^{m*}_{N}\to\bar h^m$ a.s. as $N\to\infty$ for all $m=1,\dots
,M$ and
$*\in\{+,-,\pm\}$.\vspace*{-2pt}

\item[\emph{(II)}] Assume that:\vspace*{-3pt}
\begin{enumerate}
\item[\emph{(i)}] For some $\gamma>0$, $\M[|h(\eta_m)|^{2+\gamma}]<\infty$.\vspace*{-2pt}

\item[\emph{(ii)}] $h$ is a continuous function of bounded variation on some interval
$[c_{-},c_{+}]$ and monotone on $(-\infty,c_{-}]$ and
$[c_{+},+\infty)$.
\end{enumerate}
\end{enumerate}

Then $\hat h_N^{m\pm}\to\bar h^m$ in probability.
\end{thm}

\section{Construction of the test}\label{SectTest}

We first state an asymptotic normality result for
$\hat T_N$. Denote
\[
T'(y) := \biggl( \frac{\partial}{\partial y_1} T(y),\dots,\frac{\partial
}{\partial y_d}
T(y) \biggr) \in\mathbb{R}^{L \times d}.
\]
\begin{thm}\label{Th4}
Assume that:
\begin{enumerate}
\item[\emph{(i)}] $T'(\bar g)$ exist.\vspace*{-3pt}

\item[\emph{(ii)}] The assumptions of Theorem \ref{Th2} hold.\vspace*{-3pt}

\item[\emph{(iii)}] The matrix $D=T'(\bar g)\varSigma(T'(\bar g))^T$ is nonsingular.
\end{enumerate}

Then, under $H_0$, $\sqrt{N}\hat T_N\stackrel{d}{\longrightarrow}
N(\mathbb{O}_L,D)$.
\end{thm}
For the proof, see Appendix. Note that (iii) implies the nonsingularity
of~$\varSigma$.

Now, to estimate $D$, we can use
\[
\hat D_N=T'(\tilde g_N)\tilde
\varSigma_N\bigl(T'(\tilde g_N)
\bigr)^T,
\]
where
$\tilde g_N$ is any consistent estimator for $\bar g$,
\begin{subequations}
\begin{align}
\tilde{\varSigma}^{(k,l)}_N &:= \sum
_{m=1}^M \beta_{m;N}^{k,l}
\tilde{g}_{k,l;N}^m - \sum_{r,s=1}^M
\alpha_{r,s;N}^{k,l} \tilde{g}_{k;N}^r
\bigl(\tilde{g}_{l;N}^s\bigr)^T \in
\mathbb{R}^{d_k \times d_l} ;\\
\tilde{\varSigma}_N &:= \bigl(\tilde{\varSigma}_N^{(k,l)}
\bigr)_{k,l=1,\dots,K} \in \mathbb{R}^{d \times d},
\end{align}
\end{subequations}
where $\tilde{g}_{k,l;N}^m$ is any consistent estimator for
$\bar{g}_{k,l;N}^m$. For example, we can use\vadjust{\eject}
\[
\tilde{g}_{k,l;N}^m=\hat g_{k,l;N}^{m\pm}=
{1\over N}\sum_{j=1}^N
b_{j;N}^{m\pm} g_k(\xi_{j;N})g_l(
\xi_{j;N})^T
\]
if $\XXX=\R$ and the assumptions of Theorem \ref{Th3} hold for all
$h(x)=g_l^i(x)g_k^n(x)$, $i,k=1,\dots,K$, $i=1,\dots,d_l$,
$n=1,\dots,d_k$, $g_l(x)=(g_l^1(x),\dots,g_l^{d_l}(x))^T$.

Now let the test statistic be
$\hat s_N=N (\hat T_N)^T \hat D_N^{-1}\hat T_N$. For a given
significance level~$\alpha$, the test $\pi_{N,\alpha}$ accepts $H_0$ if
$\hat s_N\le Q^{\xi_L}(1-\alpha)$ and rejects $H_0$ otherwise.

The $p$-level of the test (i.e., the attained significance level) can be
calculated as $p=1-G(\hat s_N)$, where $G$ means the CDF of
$\chi^2_L$-distribution.

\begin{thm}\label{Th5}
Let the assumptions of Theorem \ref{Th4} hold. Moreover,
assume the following:
\begin{enumerate}
\item[\emph{(i)}] $\tilde g_N$ and $\tilde{g}_{k,l;N}^m$ $(k,l=1,\dots,K$,
$m=1,\dots,M)$
are consistent
estimators for $\bar g$ and $\bar{g}_{k,l;N}^m$, respectively.

\item[\emph{(ii)}] $T'$ is continuous in some neighborhood of $\bar g$.
\end{enumerate}

Then $\lim_{N\to\infty}\pr_{H_0}\{\pi_{N,\alpha} \text{\textnormal{
rejects }}
H_0\}=\alpha$.
\end{thm}
\setcounter{example}{1}
\begin{example}[Continued] Consider testing $H_0^\mu$ by the
test $\pi_{N,\alpha}$ with $g_i(x)=x$ and
$T(y_1,\dots,y_M)=(y_1-y_2,y_2-y_3,\dots,y_{M-1}-y_M)^T$.
It is obvious that $T'(y)$ is a constant matrix of full rank.
Assume that $\D[\eta_m]>0$ for all $m=1,\dots,M$ and
$\det\varGamma\not=0$. Then $\varSigma$ is nonsingular, and so is $D$.
Thus, in this case, assumptions (i) and (iv) of Theorem \ref{Th2},
(i) and (iii) of Theorem \ref{Th4}, and (ii) of Theorem \ref{Th5}
hold.

To ensure assumption (ii) of Theorem \ref{Th2}, we need
$\M[|\eta_m|^{2+\delta}]<\infty$ for some $\delta>0$ and all
$m=1,\dots,M$. In view of Theorem \ref{Th1}, this assumption also
implies the consistency of $\hat g_N$ and $\hat g_{kl;N}^m$. If one
uses $\hat g_N^{\pm}$ and $\hat g_{kl;N}^{m\pm}$ as estimators
$\tilde g_N$ and $\tilde g_{kl;N}^m$ in $\hat D_N$, then a more
restrictive assumption $\M[|\eta_m|^{4+\delta}]<\infty$ is needed
to ensure their consistency by Theorem \ref{Th3}.
\end{example}

\section{Numerical results}\label{SectNum}

\subsection{Simulation study}\label{SectSimul}
To access the proposed test performance on samples of moderate size,
we conducted a small simulation study. Three-component mixtures
were analyzed ($M=3$) with Gaussian components $F_m\sim
N(\mu_m,\sigma_m^2)$. The concentrations were generated as\allowbreak
$p_{j,N}^m=\zeta_{j;N}^m/s_{j;N}$, where $\zeta_{j;N}^m$ are
independent, uniformly distributed on $[0,1]$ random variables, and
$s_{j;N}=\sum_{m=1}^M\zeta_{j;N}^m$. In all the experiments, 1000
samples were generated for each sample size
$N=$ 50, 100, 250, 500, 750, 1000, 2000, and 5000. Three
modifications of $\pi_{N;\alpha}$ test were applied to each
sample. In the first modification, (ss), simple estimators were used to
calculate both $\hat T_N$ and $\hat D_N$. In the second
modification, (si), simple estimators were used in $\hat T_N$, and the
improved ones were used in $\hat D_N$. In the last modification
(ii), improved estimators were used in $\hat T_N$ and $\hat D_N$.
Note that the modification (ii) has no theoretical justification
since, as far as we know, there are no results on \xch{the limit
distribution}{the imit
distribution} of $\sqrt{N}(\hat g_N^{\pm}-\bar g)$.

All tests were used with the nominal significance level
$\alpha=0.05$.\vadjust{\eject}

In the figures, frequencies of errors of the
tests are presented. In the
plots, $\Box$ corresponds to (ss), $\triangle$ to (si), and
$\circ$ to (ii) modification.

\begin{figure}[t]
\includegraphics{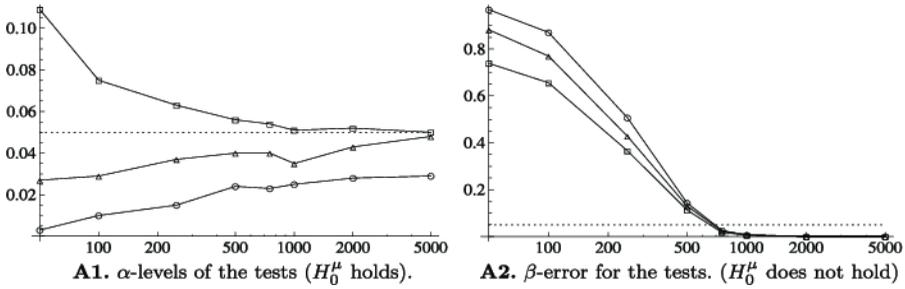}
\caption{Testing homogeneity of means ($H_0^\mu$)}\label{Fig1}
\end{figure}

\def\theexperiment{A\arabic{experiment}}
\begin{experiment}\label{experimenta1} In this experiment, we consider testing
the mean homogeneity hypothesis $H_0^{\mu}$. The means were
taken
$\mu_m=0$, $m=1,2,3$, so $H_0^{\mu}$ holds. To shadow the equality of
means,
different variances
of components were taken, namely $\sigma_1^2=1$, $\sigma_2^2=4$, and
$\sigma_3^2=9$. The resulting
first-type error frequencies are presented on the left panel
of Fig.~\ref{Fig1}. For the (ss) test, for small~$N$,
there were 1.4\% cases of incorrect covariance matrix estimates
($\hat D_N$~was not positive definite). Incorrect
estimates were absent for large $N$.
\end{experiment}

\begin{experiment} Here we also tested $H_0^{\mu}$ for
components with the same variances as in \ref{experimenta1}. But
$\mu_1=2$ and $\mu_2=\mu_3=0$, so $H_0^\mu$ does not hold. The
frequencies of the second-type error are presented on the right panel
of Fig.~\ref{Fig1}. The percent of incorrect estimates $\hat D_N$ is
1.6\%
for (ss) and small $N$.
\end{experiment}

\setcounter{experiment}{0}
\def\theexperiment{B\arabic{experiment}}
\begin{experiment}\label{experimentb1} In this and the next experiment, we tested
$H_0^\sigma$: $\sigma_1^2=\sigma_2^2$. The data were generated with
$\mu_1=0$, $\mu_2=3$, $\mu_3=-2$, $\sigma_1^2=\sigma_2^2=1$,
and $\sigma_2^2=4$, so $H_0^{\sigma}$ holds. The frequencies of the
first
-type error are
presented on the left panel of Fig. \ref{Fig2}. The percent of
incorrect $\hat D_N$ in (ss) varies from 19.4\% for small $N$ to
0\% for large~$N$.
\end{experiment}

\begin{experiment} Now $\mu_m$ and $\sigma_3^2$ are the same as in
\ref{experimentb1}, but $\sigma_1^2=1$ and $\sigma_2^2=4$, so $H_0^\sigma$
does not hold. The frequencies of the second-type error are
presented on the left panel of Fig. \ref{Fig2}. The percent of
incorrect $\hat D_N$ in (ss) was 15.5\% for small $N$ and decreases to
0\% for large $N$.
\end{experiment}

\begin{figure}[t!]
\includegraphics{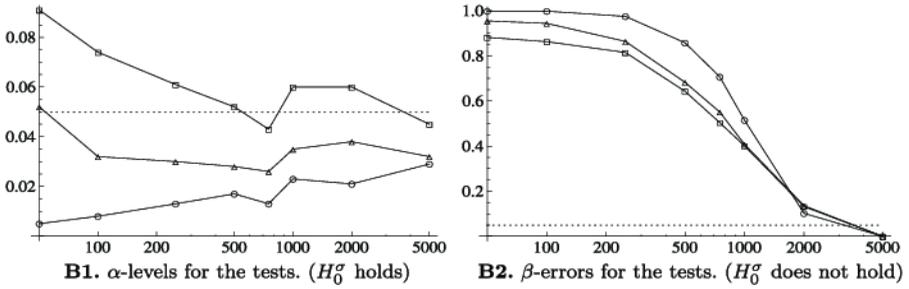}
\caption{Testing equality of variances ($H_0^\sigma$)}\label{Fig2}
\end{figure}
The presented results show reasonable agreement of the observed
significance levels of the tests to their nominal level 0.05 when
the sample sizes were larger then 500. The power of the tests
increases to 1 as the sample sizes grow. It is interesting to
note that the (ii) modification, although theoretically not
justified, demonstrates the least first-type error and rather
good power. From these results the (si) modification of the test
seems the most prudent one.

\subsection{Example of a sociological data analysis}

Consider the data discussed in \cite{Maiboroda:AdaptMVC}. It
consists of two parts. The first part is the data from the
Four-Wave
World Values Survey (FWWVS) held in Ukraine by the European Values
Study Foundation (\url{www.europeanvalues.nl}) and World Values Survey
Association (\url{www.worldvaluessurvey.org}) in 2006. They contain
answers of $N=4006$ Ukrainian respondents on different questions
about their social status and attitudes to different human values.
We consider here the level of satisfaction of personal income
(subjective income) as our variable of interest $\xi$, so
$\xi_{j;N}$ is the subjective income of the $j$th respondent.

Our aim is to analyze differences in the distribution of $\xi$ on
populations of adherents of different political parties. Namely,
we use the data on results of Ukrainian Parliament elections held
in 2006. 46 parties took part in the elections. The voters could
also vote against all or not to take part in the voting. We
divided all the population of Ukrainian voters into three large
groups (political subpopulations): $\PP_1$ which contains
adherents of
the Party of Regions (\textit{PR}, 32.14\% of votes), $\PP_2$
of Orange Coalition supporters (\textit{OC} which consisted of
``BJUT'' and
``NU'' parties, 36.24\%), and $\PP_3$ of all others, including the
persons who voted against all or did not take part in the pool
(\textit{Other}).

Political preferences of respondents are not available in the FWWVS
data, so we used official results of the elections by 27 regions
of Ukraine (see the site of Ukrainian Central Elections
Commission \url{www.cvk.gov.ua}) to estimate the concentrations
$p_{j;N}^m$ of the considered political subpopulations in the
region where the $j$th respondent voted.

Means and variances of $\xi$ over different subpopulations were
estimated by the data (see Table \ref{Tab0}).
\begin{table}
\caption{Means ($\mu$) and variances ($\sigma^2$) for the
subjective income distribution on different political populations
}\label{Tab0}
\begin{tabular}{llll}\hline
& \textit{PR} & \textit{OC} & \textit{Other} \\ \hline
$\mu$ & 2.31733 & 2.65091 & 4.44504\\
$\mu^{+}$ & 2.45799 & 2.64187& 4.44504\\
$\sigma^2$ & 0.772514 & 4.85172 & 4.93788\\
$\sigma^{2+}$ & 2.09235 & 4.7639 & 4.93788\\ \hline
\end{tabular}
\end{table}
Different tests were performed to test
their differences. The results are presented in the Table
\ref{Tab1}. Here $\mu_m$ means
the expectation, and $\sigma_m^2$ means the variance of $\xi$ over
the $m$th subpopulation. Degrees of freedom for the limit
$\chi^2$ distribution are placed in the ``df'' column.

These results show that the hypothesis of
homogeneity of all variances must be definitely rejected. The
variances of $\xi$ for
\textit{PR} and \textit{OC} adherents are different,
but the tests failed
to observe significant differences in the pairs of variances
\textit{PR-Other} and \textit{OC-Other}. For the means, all the
tests agree that \textit{PR} and \textit{OC} has the same mean
$\xi$, whereas the mean of \textit{Other} is different from the common
mean of \textit{PR} and \textit{OC}.

\begin{table}
\caption{Test statistics and $p$-values for hypotheses on subjective
income distribution}\label{Tab1}
\begin{tabular}{lllll}\hline
Hypotheses & ss & si & ii & df \\\hline
$\mu_1=\mu_2=\mu_3$ & 11.776 & 10.8658 & 8.83978 & 2\\\smallskip
$p$-value & 0.00277252 & 0.0043704 & 0.0120356 & \\
$\mu_1=\mu_2$ & 2.15176 & 2.04539 & 0.621483 & 1 \\\smallskip
$p$-value & 0.142407 & 0.152668 & 0.430497 & \\
$\mu_1=\mu_3$ & 10.7076 & 10.0351 & 8.75216 & 1 \\\smallskip
$p$-value & 0.00106696 & 0.00153585 & 0.00309236 & \\
$\mu_2=\mu_3$ & 7.40835 & 7.10653 & 7.17837 & 1\\\smallskip
$p$-value & 0.00649218 & 0.00768036 & 0.00737877 & \\
$\sigma_1^2=\sigma_2^2=\sigma_3^2$ & 15.8317 & 14.786 & 6.40963  & 2\\\smallskip
$p$-value & 0.000364914 & 0.000615547 & 0.0405664 & \\
$\sigma_1^2=\sigma_2^2$ & 14.7209 & 13.8844 & 5.95528 & 1\\\smallskip
p-level & 0.000124657 & 0.000194405 & 0.0146733 & \\
$\sigma_1^2=\sigma_3^2$ & 1.92166 & 1.77162 & 0.826778 & 1 \\\smallskip
$p$-level & 0.165674 & 0.183182 & 0.363206 & \\
$\sigma_2^2=\sigma_3^2$ & 0.000741088 & 0.00072198 & 0.00294353 & 1\\
$p$-level &0.978282 & 0.978564 & 0.956733 & \\\hline
\end{tabular}
\end{table}

\section{Concluding remarks}

We developed a technique that allows one to construct testing
procedures for different hypotheses on functional moments of
mixtures with varying concentrations. This technique can be
applied to test the homogeneity of means or variances (or both) of some
components of the mixture. Performance of different modifications of
the test
procedure is compared in a small simulation study. The (ss) modification
showed the worst first-type error and the highest power. The (ii)
test has the best first-type error and the worst power.
It seems that
the (si) modification can be recommended as a golden mean.

\section*{Acknowledgement} The authors are thankful
to the anonymous referee for fruitful comments.

\appendix
\section{Appendix}\label{appendix}

\begin{proof}[Proof of Theorem \ref{Th2}]
Note that
$S_N=\sqrt{N}(\hat g_N-\bar g)=\sum_{j=1}^N \zeta_{j;N}$,
where
\[
\zeta_{j;N}={1\over\sqrt{N}} \bigl(a_{j;N}^i
\bigl(g_1(\xi_{j;N})-M\bigl[g_i(\xi
_{j;N})\bigr]\bigr)^T, i=1,\dots,K \bigr)^T.
\]
We will apply the CLT with the Lyapunov condition (see Theorem 11 from
Chapter 8 and Remark 4 in Section 4.8 in \cite{BorovkovPr}) to $S_N$.
It is readily seen that $\zeta_{j:N}$\xch{, $j=1,\dots,N$}{$j=1,\dots,N$}, are
independent for fixed $N$ and $\M[\zeta_{j;N}]=0$.\vadjust{\eject}

Let $\varSigma_{j;N}=\cov(\zeta_{j;N})$. Then $\varSigma_{j;N}$
consists of the blocks
\begin{align*}
\varSigma_{j;N}^{(k,l)}&= {1\over N}
a_{j;N}^k a_{j;N}^l\bigl(\M
\bigl[g_k(\xi_{j;N}) \bigl(g_l(
\xi_{j;N})\bigr)^T\bigr] - \M\bigl[g_k(
\xi_{j;N})\bigr]\M\bigl[g_l(\xi_{j;N})
\bigr]^T\bigr)\\
&= {1\over N} a_{j;N}^k a_{j;N}^l
\Biggl( \sum_{m=1}^M p_{j;N}\bar
g_{k,l}^m - \Biggl(\sum_{m=1}^M
p_{j;N}\bar g_{k}^m \Biggr) \Biggl(\sum
_{m=1}^M p_{j;N}\bar g_{l}^m
\Biggr)^T \Biggr).
\end{align*}
It is readily seen that $\sum_{j=1}^N
\varSigma_{j;N}^{(k,l)}\to\varSigma^{(k,l)}$
as $N\to\infty$. So
\begin{equation}
\label{Eq1ap} \cov S_N\to\varSigma\quad\text{ as } N\to\infty.
\end{equation}
To apply the CLT, we only need to verify the Lyapunov condition
\begin{equation}
\label{Eq2ap} \sum_{j=1}^N \M\bigl[\big|
\zeta_{j;N}\big|^{2+\delta}\bigr]\to0\quad \text{ for some } \delta>0.
\end{equation}
Note that assumption (iii) implies
\[
\sup_{1\le j\le N, 1<\le m\le M, N>N_0} \big|a_{j;N}^m\big|<C_1
\]
for some $N_0$ and $C_1<\infty$; thus,
\begin{equation}
\label{Eq3ap} \sum_{j=1}^N \M\bigl[\big|
\xi_{j;N}\big|^{2+\delta}\bigr]\le \sum_{j=1}^N
{C_1^{2+\delta}\over N^{1+\delta/2}} \M\bigl[\big|g(\xi_{j;N})\big|^{2+\delta}\bigr],
\end{equation}
where
$
g(x)=(g_1(x)^T,\dots,g_k(x)^T)
$.
Since
$
|g(\xi_{j;N})|^2=\sum_{k=1}^K |g_k(\xi_{j;N})|^2
$
and, by the H\"{o}lder inequality,
\[
\big|g(\xi_{j;N})\big|^{2+\delta}\le K^{\delta/2}\sum
_{k=1}^K\big|g_k(\xi_{j;N})\big|^{2+\delta},
\]
we obtain
\begin{align*}
\M\bigl[\big|g(\xi_{j;N})\big|^{2+\delta}\bigr]&\le K^{\delta/2}\sum
_{k=1}^K \M\bigl[\big|g(\xi_{j;N})\big|^{2+\delta}
\bigr]\\
&= K^{\delta/2}\sum_{k=1}^K\sum
_{m=1}^M p_{j;N}^m
\M\big|g(\eta_m)\big|^{2+\delta}<C_2<\infty,
\end{align*}
where the constant $C_2$ does not depend on $j$ and $N$.
This, together with (\ref{Eq3ap}), yields (\ref{Eq2ap}). By the
CLT we obtain
$S_N\stackrel{d}{\longrightarrow}N(\mathbb{O},\varSigma)$.
\end{proof}

\begin{proof}[Proof of Theorem \ref{Th3}. Part (I)]
Since $\hat F_{m;N}^{+}(x)$ is piecewise constant and $F_m(x)$ is
nondecreasing, the $\sup_x$ of $|\hat F_{m;N}^{+}(x)-F_m(x)|$ can be
achieved only at jump points of $\hat F_{m;N}^{+}(x)$. But
$\hat F_{m;N}^{+}(x)\ge\hat F_{m;N}(x)$ for all $x$, and if $x$
is a jump point, then
\[
\hat F_{m;N}(x-)\le\hat F_{m;N}^{+}(x-)\le\hat
F_{m;N}^{+}(x)\le\hat F_{m;N}(x).
\]
Therefore,
\[
\sup_{x\in\R}\big|\hat F_{m;N}^{+}(x)-F_m(x)\big|
\le \sup_{x\in\R}\big|\hat F_{m;N}(x)-F_m(x)\big|.
\]
Similarly,
\[
\sup_{x\in\R}\big|\hat F_{m;N}^{*}(x)-F_m(x)\big|
\le \sup_{x\in\R}\big|\hat F_{m;N}(x)-F_m(x)\big|.
\]
By the Glivenko--Cantelli-type theorem for weighted empirical
distributions (which can be derived, e.g., as a
corollary of Theorem 2.4.2 in \cite{Maiboroda2003})
\[
\sup_{x\in\R}\big|\hat F_{m;N}(x)-F_m(x)\big|\to0
\quad\text{ a.s. as } N\to\infty
\]
if $\sup_{j=1\dots,N;N>N_0}|a_{j;N}^m|<\infty$. The latter condition is
fulfilled
since $\det\varGamma\not=0$. Thus,
\begin{equation}
\label{Eq4ap} \sup_{x\in\R}\big|\hat F_{m;N}^*(x)-F_m(x)\big|
\to0\quad \text{ a.s. as } N\to\infty.
\end{equation}
For any $h:\R\to\R$ and any interval $A\subseteq\R$, let $V_A(h)$
be the variation of $h$ on~$A$. Take $A=(c_{-},c_{+})$. Then, under
the assumptions of the theorem,
\begin{align*}
\big|\hat h_N^{m*}-\bar h_m\big|&= \bigg|\int
_A h(x)d\bigl(\hat F_{m;N}^*(x)-F_m(x)
\bigr)\bigg|\\
&\le\sup_{x\in A}\big|\hat F_{m;N}(x)-F_m(x)\big|
\bcdot V_A(h)\to0\quad \text{ a.s. as } N\to\infty.
\end{align*}

\textbf{Part (II).} Note that if the assumptions of this part
hold for some $A=(c_{-},c_{+})$, then they will also hold for any new
$c_{-}$, $c_{+}$ such that $A\subset(c_{-},c_{+})$. Thus, we may
assume that $F_m(c_{-})<1/4$ and $F_m(c_{+})>3/4$.

Consider the random event
$B_N^1=\{\hat F_{m;N}^{\pm}(x)=\tilde F_{m;N}^{+}(x) \text{ for all }
x\le c_{-}\}$. Then (\ref{Eq4ap}) implies
$\pr\{B_N^1\}\to 1$ as $N\to\infty$.

We bound
\begin{equation}
\label{EqHap} \big|\hat h_N^{m\pm}-\bar h^m\big|\le
J_1+J_2+J_3,
\end{equation}
where
\begin{align*}
J_1&=\Biggl\llvert \int_{-\infty}^{c_{-}} h(x)d
\bigl(\hat F_{m;N}^{\pm}(x)-F_m(x)\bigr) \Biggr
\rrvert ,\\
J_2&=\Biggl\llvert \int_{c_{-}}^{c_{+}} h(x)d
\bigl(\hat F_{m;N}^{\pm}(x)-F_m(x)\bigr) \Biggr
\rrvert ,\\
J_3&=\Biggl\llvert \int_{c_{+}}^{+\infty} h(x)d
\bigl(\hat F_{m;N}^{\pm}(x)-F_m(x)\bigr) \Biggr
\rrvert .
\end{align*}
Then $J_2\stackrel{P}{\longrightarrow}0$ as in Part (I).

Let us assume that the event $B_N^1$ occurred and bound
\begin{equation}
\label{Eq6ap} J_1\le\Biggl\llvert \int_{-\infty}^{c_{-}}
h(x)d\bigl(\tilde F_{m;N}^{+}(x)-F_m(x)\bigr)
\Biggr\rrvert .
\end{equation}
If $h(x)$ is bounded as $x\to-\infty$, then we can take
$c_{-}=-\infty$ and obtain $J_1=0$. Consider the case of
unbounded $h$. Since $h$ is monotone, we have $h(x)\to+\infty$ or
$h(x)\to+\infty$ as $x\to-\infty$. We will consider the first
case; the reasoning for the second one is analogous. Thus, $h(x)\to
+\infty$
as $x\to-\infty$, and we can take $h(x)>0$ for $x<c_{-}$.

By the inequality (16) in \cite{MK2005},
\begin{equation}
\label{Eq7ap} \pr\Bigl[\sup_{t<x}\big|\tilde F_{m;N}^{+}(t)-F_m(t)\big|>
\varepsilon\Bigr] \le C_1\bigl(\bar F^2(x)
\varepsilon^{-4}N^{-2}+\bar F(x)\varepsilon^{-2}N^{-1}
\bigr),
\end{equation}
where $\bar F(x)=\sum_{m=1}^M F_m(x)$, $C_1<\infty$.

Let us take $x_0,\dots,x_n,\dots$ such that $h(x_j)=2^jh(c_{-})$.
By assumption (ii) and the Chebyshev inequality,
\[
\bar F(x)=\sum_{m=1}^M\pr[
\eta_m<x] \le \sum_{m=1}^M
h^{-2-\gamma}(x)\M\bigl[\big|h(\eta_m)\big|^{2+\gamma}\bigr],
\]
and
\[
\bar F(x_j)\le D 2^{-(2+\gamma)j}
\]
for some $D<\infty$.

Let $\varepsilon_j=2^{-(1+\gamma/4)j}N^{-1/4}$. Then
by (\ref{Eq7ap})
\[
\pr\Bigl[\sup_{t<x_j}\big|\tilde F_{m;N}^{+}(t)-F_m(t)\big|>
\varepsilon_j\Bigr] \le C_2 \bigl( 2^{-\gamma j}N^{-1}+2^{-\gamma j/2}N^{-1/2}
\bigr)
\]
for some $C_2<\infty$.
Denote $B_N^2=\cap_j\{\sup_{t<x_j}|\tilde
F_{m;N}^{+}(t)-F_m(t)|\le\varepsilon_j\}$.
Then
\[
\pr\bigl[B_N^2\bigr]\ge1-\sum
_{j=1}C_2 \bigl( 2^{-\gamma j}N^{-1}+2^{-\gamma j/2}N^{-1/2}
\bigr)\ge1-C_3N^{-1}-C_4N^{-1/2}\to1
\]
as $N\to\infty$.
Now,
$J_{1}=|\int_{-\infty}^{c_{-}}h(x)d(\tilde F_{m;N}^{+}(x)-F_m(x))|$.
If $B_N^1$ and $B_N^2$ occur, then
\begin{align*}
J_{1}&\le \Biggl\llvert \sum_{j=0}^N
\int_{x_{j+1}}^{x_j} \big|\tilde F_{m;N}^{+}(x)-F_m(x)\big|h(dx)
\Biggr\rrvert\\
&\le\sum_{j=0}^N h(c_{-})
2^{-(1+\gamma/4)j}N^{-1/4}\bigl(2^{j+1}-2^j\bigr)\le
C_5 N^{-1/4}.
\end{align*}
Thus,
$
\pr[J_1\le C_5 N^{-1/4}]\ge\pr[B_N^1\cap B_N^2]\to1
$
and $J_1\stackrel{P}{\longrightarrow}0$.

Similarly, $J_3\stackrel{P}{\longrightarrow}0$.

Combining these bounds with (\ref{EqHap}), we accomplish the proof.
\end{proof}

\begin{proof}[Proof of Theorem \ref{Th4}]
This theorem is a simple corollary of Theorem \ref{Th2} and the
continuity theorem for weak convergence (Theorem 3B in Chapter 1 of
\cite{BorovkovMs}). \end{proof}

\begin{proof}[Proof of Theorem \ref{Th5}]
Since $\tilde g_N$ and $\tilde g_{kl;N}^m$ are consistent,
$\tilde
\varSigma_N\stackrel{P}{\longrightarrow}\varSigma$. Similarly,
the continuity of~$T'$ and consistency of $\tilde g_N$ imply
$T'(\tilde g_N)\stackrel{P}{\longrightarrow}T'(\bar g)$. Then,
with \mbox{$\det D\not=0$} in mind, we obtain
$\hat D_N^{-1}\stackrel{P}{\longrightarrow} D^{-1}$.

Denote $\tilde s_N=N\hat T_N^TD^{-1}\hat T_N$. By Theorem \ref{Th4} and
the continuity theorem,\allowbreak \mbox{$\tilde
s_N\stackrel{d}{\longrightarrow}\chi^2_L$.}
By Theorem
\ref{Th4} $\sqrt{N}\hat T_N$ is stochastically bounded. Thus,
\[
\big|\tilde s_N-\hat s_N\big|=\big|\sqrt{N}\hat T_N^T
\bigl(D^{-1}-\hat D_N^{-1}\bigr) (\sqrt{N}\hat
T_N)\big|\stackrel{P} {\longrightarrow} 0.
\]
Therefore, $\hat s_N$ converges in distribution to the same limit as
$\tilde s_N$, that is, to $\chi_L^2$. \end{proof}


%

\begin{thebibliography}{99}

\bibitem{Autin:TestDensities}
\begin{barticle}
\bauthor{\bsnm{Autin}, \binits{F.}},
\bauthor{\bsnm{Pouet}, \binits{Ch.}}:
\batitle{Test on the components of mixture densities}.
\bjtitle{Stat. Risk. Model.}
\bvolume{28}(\bissue{4}),
\bfpage{389}--\blpage{410}
(\byear{2011}).
\bcomment{\MR{2877572}}.
doi:\doiurl{10.1524/strm.2011.1065}
\end{barticle}
\OrigBibText
Autin, F., Pouet, Ch.: Test on the components of mixture
densities. Statistics \& Risk Modelling \textbf{28}, No 4,
389--410 (2011)
\endOrigBibText
\bptok{structpyb}
\endbibitem

\bibitem{BorovkovPr}
\begin{bbook}
\bauthor{\bsnm{Borovkov}, \binits{A.A.}}:
\bbtitle{Probability Theory}.
\bpublisher{Gordon and Breach Science Publishers},
\blocation{Amsterdam}
(\byear{1998}).
\bcomment{\MR{1711261}}
\end{bbook}
\OrigBibText
Borovkov, A.A.: Probability Theory. Gordon and Breach
Science Publishers, Amsterdam (1998)
\endOrigBibText
\bptok{structpyb}
\endbibitem

\bibitem{BorovkovMs}
\begin{bbook}
\bauthor{\bsnm{Borovkov}, \binits{A.A.}}:
\bbtitle{Mathematical statistics}.
\bpublisher{Gordon and Breach Science Publishers},
\blocation{Amsterdam}
(\byear{1998}).
\bcomment{\MR{1712750}}
\end{bbook}
\OrigBibText
Borovkov, A.A.: Mathematical statistics. Gordon and Breach Science Publishers,
Amsterdam (1998)
\endOrigBibText
\bptok{structpyb}
\endbibitem

\bibitem{Johnson}
\begin{barticle}
\bauthor{\bsnm{Johnson}, \binits{N.L.}}:
\batitle{Some simple tests of mixtures with symmetrical components}.
\bjtitle{Commun. Stat.}
\bvolume{1}(\bissue{1}),
\bfpage{17}--\blpage{25}
(\byear{1973}).
\bcomment{\MR{0315824}}.
doi:\doiurl{10.1080/03610927308827004}
\end{barticle}
\OrigBibText
Johnson, N.L.: Some simple tests of mixtures with symmetrical
components. Commun. Statist. \textbf{1}, No 1., 17--25 (1973)
\endOrigBibText
\bptok{structpyb}
\endbibitem

\bibitem{Liu}
\begin{barticle}
\bauthor{\bsnm{Liu}, \binits{X.}},
\bauthor{\bsnm{Pasarica}, \binits{C.}},
\bauthor{\bsnm{Shao}, \binits{Y.}}:
\batitle{Testing homogeneity in gamma mixture model}.
\bjtitle{Scand. J. Stat.}
\bvolume{20},
\bfpage{227}--\blpage{239}
(\byear{2003}).
\bcomment{\MR{1965104}}.
doi:\doiurl{10.1111/1467-9469.00328}
\end{barticle}
\OrigBibText
Liu, X., Pasarica, C., Shao, Y.: Testing homogeneity in gamma mixture
model. Scand. J. Statist. \textbf{20} 227--239 (2003)
\endOrigBibText
\bptok{structpyb}
\endbibitem

\bibitem{Maiboroda2000}
\begin{barticle}
\bauthor{\bsnm{Maiboroda}, \binits{R.E.}}:
\batitle{A test for the homogeneity of mixtures with varying concentrations}.
\bjtitle{Ukr. J. Math.}
\bvolume{52}(\bissue{8}),
\bfpage{1256}--\blpage{1263}
(\byear{2000}).
\bcomment{\MR{1819720}}.
doi:\doiurl{10.1023/A:1010305121413}
\end{barticle}
\OrigBibText
Maiboroda, R.E.: A Test for the homogeneity of mixtures with varying
concentrations.
Ukr. J. Math., \textbf{52}, No 8, 1256--1263 (2000)
\endOrigBibText
\bptok{structpyb}
\endbibitem

\bibitem{Maiboroda2003}
\begin{bbook}
\bauthor{\bsnm{Maiboroda}, \binits{R.}}:
\bbtitle{Statistical Analysis of Mixtures}.
\bpublisher{Kyiv University Publishers},
\blocation{Kyiv}
(\byear{2003}).
\bcomment{(in Ukrainian)}
\end{bbook}
\OrigBibText
Maiboroda, R.: Statistical Analysis of Mixtures. Kyiv
University Publishers, Kyiv (in Ukrainian) (2003)
\endOrigBibText
\bptok{structpyb}
\endbibitem

\bibitem{Maiboroda:StatisticsDNA}
\begin{barticle}
\bauthor{\bsnm{Maiboroda}, \binits{R.}},
\bauthor{\bsnm{Sugakova}, \binits{O.}}:
\batitle{Statistics of mixtures with varying concentrations with application to DNA microarray data analysis}.
\bjtitle{J. Nonparametr. Stat.}
\bvolume{24}(\bissue{1}),
\bfpage{201}--\blpage{205}
(\byear{2012}).
\bcomment{\MR{2885834}}.
doi:\doiurl{10.1080/10485252.2011.630076}
\end{barticle}
\OrigBibText
Maiboroda, R., Sugakova, O.: Statistics of mixtures with varying
concentrations with application to DNA microarray data analysis.
Journal of Nonparametric Statistics. \textbf{24}, No 1, 201--205
(2012)
\endOrigBibText
\bptok{structpyb}
\endbibitem

\bibitem{Maiboroda:AdaptMVC}
\begin{barticle}
\bauthor{\bsnm{Maiboroda}, \binits{R.E.}},
\bauthor{\bsnm{Sugakova}, \binits{O.V.}},
\bauthor{\bsnm{Doronin}, \binits{A.V.}}:
\batitle{Generalized estimating equations for mixtures with varying concentrations}.
\bjtitle{Can. J. Stat.}
\bvolume{41}(\bissue{2}),
\bfpage{217}--\blpage{236}
(\byear{2013}).
\bcomment{\MR{3061876}}.
doi:\doiurl{10.1002/cjs.11170}
\end{barticle}
\OrigBibText
Maiboroda, R.E., Sugakova, O.V., Doronin, A.V.: Generalized
estimating equations for mixtures with varying concentrations. The
Canadian Journal of Statistics \textbf{41}, No 2, 217--236 (2013)
\endOrigBibText
\bptok{structpyb}
\endbibitem

\bibitem{MK2005}
\begin{barticle}
\bauthor{\bsnm{Majboroda}, \binits{R.}},
\bauthor{\bsnm{Kubajchuk}, \binits{O.}}:
\batitle{Improved estimates for moments by observations from mixtures}.
\bjtitle{Theory Probab. Math. Stat.}
\bvolume{70},
\bfpage{83}--\blpage{92}
(\byear{2005}).
doi:\allowbreak\doiurl{10.1090/S0094-9000-05-00642-3}
\end{barticle}
\OrigBibText
Majboroda, R., Kubajchuk, O.: Improved estimates for moments by
observations from mixtures.
Theory Probab. Math. Stat. \textbf{70}, 83--92 (2005)
\endOrigBibText
\bptok{structpyb}
\endbibitem

\bibitem{McLachlan}
\begin{bbook}
\bauthor{\bsnm{McLachlan}, \binits{G.J.}},
\bauthor{\bsnm{Peel}, \binits{D.}}:
\bbtitle{Finite Mixture Models}.
\bpublisher{Wiley-Interscience},
\blocation{New York}
(\byear{2000}).
\bcomment{\MR{1789474}}.
doi:\doiurl{10.1002/0471721182}
\end{bbook}
\OrigBibText
McLachlan, G.J., Peel, D. Finite Mixture Models.
Wiley-Interscience (2000)
\endOrigBibText
\bptok{structpyb}
\endbibitem

\bibitem{Shcherbina}
\begin{barticle}
\bauthor{\bsnm{Shcherbina}, \binits{A.}}:
\batitle{Estimation of the mean value in the model of mixtures with varying concentrations}.
\bjtitle{Theory Probab. Math. Stat.}
\bvolume{84},
\bfpage{151}--\blpage{164}
(\byear{2012}).
\bcomment{\MR{2857425}}.
doi:\doiurl{10.1090/S0094-9000-2012-00866-1}
\end{barticle}
\OrigBibText
Shcherbina, A.: Estimation of the mean value in the model of
mixtures with varying concentrations.
Theor. Probability and Math. Statist.
\textbf{84} 151--164. (2012)
\endOrigBibText
\bptok{structpyb}
\endbibitem

\bibitem{Titterington:FiniteMixture}
\begin{bbook}
\bauthor{\bsnm{Titterington}, \binits{D.M.}},
\bauthor{\bsnm{Smith}, \binits{A.F.}},
\bauthor{\bsnm{Makov}, \binits{U.E.}}:
\bbtitle{Analysis of Finite Mixture Distributions}.
\bpublisher{Wiley},
\blocation{New York}
(\byear{1985})
\end{bbook}
\OrigBibText
Titterington, D.M., Smith, A.F., Makov, U.E. Analysis of Finite
Mixture Distributions. Wiley, New York (1985)
\endOrigBibText
\bptok{structpyb}
\endbibitem



\end{thebibliography}
\end{document}